\documentclass[letterpaper, 10 pt, conference]{ieeeconf}

\IEEEoverridecommandlockouts                              
\usepackage{amsmath,amssymb,amsfonts,xcolor}

\usepackage{stmaryrd}
\usepackage{mathrsfs}
\usepackage{array}
\usepackage{multirow}
\usepackage{bm}
\usepackage{tikz}
\usepackage{pgfplots}
\usepackage{cite}
\usepgfplotslibrary{groupplots}
\pgfplotsset{compat=1.18}



\usepackage{amsthm}

\def\idty{{\operatorname{Id}}}

\def\cF{{\mathcal F}}
\def\cZ{{\mathcal Z}}
\def\N{{\mathbb N}}

\def\R{{\mathbb R}}
\def\S{{\mathbb S}}

  \def\cG{{\mathcal G}} \def\cM{{\mathcal M}} \def\cS{{\mathcal S}}    \def\cN{{\mathcal N}}  \def\cC{{\mathcal C}}  \def\cI{{\mathcal I}}  \def\cU{{\mathcal U}}      \def\cE{{\mathcal E}}     \def\cF{{\mathcal F}}   \def\cR{{\mathcal R}} \def\cX{{\mathcal X}} \def\cY{{\mathcal Y}}  \def\cZ{{\mathcal Z}}

\theoremstyle{plain}
\newtheorem{theorem}{Theorem}[section] 
\newtheorem{proposition}[theorem]{Proposition}
\newtheorem{lemma}[theorem]{Lemma}

\theoremstyle{definition}

\newtheorem{definition}{Definition}
\newtheorem*{notations}{Notation}

\newtheorem{remark}[theorem]{Remark}

\newtheorem{assumption}[theorem]{Assumptions}

\usepackage{hyperref}
\hypersetup{
	pdfauthor={Cyprien Tamekue, Zongxi Yu, and ShiNung Ching},
	pdfsubject={Paper manuscrit},
	pdftitle={Minimum-energy steering control for control-affine systems via
		nonlinear controllability gramian},
	pdfkeywords={Control Synthesis, Nonlinear Control, },
	colorlinks=true,
	linkcolor=blue,    
	citecolor=blue,    
	urlcolor=blue      
}

\title{\LARGE \bf Minimum-Energy Control For Control-Affine Systems
}

\author{Cyprien Tamekue, Zongxi Yu and ShiNung Ching
	\thanks{This work is partially supported by grant R21MH132240 from the US National Institutes of Health to SC.}
	\thanks{The authors are with the Department of Electrical and Systems Engineering, Washington University in St. Louis, St. Louis, 63130, MO, USA,
		{\tt\small e-mail: \{cyprien,y.zongxi,shinung\}@wustl.edu.}}%
}

\begin{document}

	\maketitle
	\thispagestyle{empty}
	\pagestyle{empty}

	\begin{abstract}
		Fixed-horizon minimum-energy control for control-affine systems rarely admits closed-form formulas beyond the linear setting. In this article, under the assumption of unconstrained controls, we derive a fixed-horizon minimum-energy control for a broad class of control-affine systems as a fixed point of a Lagrange multiplier synthesis map. Compared with our earlier general steering synthesis, the present construction is based on a generally non-symmetric Gramian-like matrix and solves the associated $L^2$-energy problem. More specifically, under a comparison estimate criterion and a self-mapping assumption, we prove that the synthesis map has a unique fixed point in the feasible coercive classes, which can be computed via a standard Picard iteration. As a demonstration of concept, we use uniform complete controllability results for linear time-varying systems to derive a bracket-generating condition ensuring complete controllability for time-dependent planar control-affine systems with scalar input, and our minimum energy framework applies under standing assumptions. Special treatment for the unicycle kinematic model is also provided, and numerical examples illustrate the approach’s effectiveness.
	\end{abstract}
	


	\section{Introduction}
	\label{sec:introduction}
	Since Kalman's formal introduction of controllability for linear control systems~\cite{kalman1960general}, control analysis and synthesis for controlled dynamical systems have remained a central topic due to pervasive applications in engineering and science~\cite{coron2007control}. In the linear time-invariant (LTI) and linear time-varying (LTV) settings, controllability admits algebraic characterizations~\cite{silverman1967controllability} and yields constructive Gramian formulas, including fixed-horizon minimum-energy controls.
	
	For nonlinear control-affine systems, the situation is intrinsically geometric: reachability is governed by the endpoint map and by directions generated through Lie brackets and orbit-type constructions. Classical results such as the rank viewpoint~\cite{hermann2003nonlinear}, the general controllability framework~\cite{sussmann1972nonlinear}, the local controllability theorem~\cite{sussmann1987local}, and Nagano--Sussmann's orbit perspective~\cite{agrachev2013geometric} explain why no fixed, trajectory-independent Gramian can play the role of the LTI Gramian in general nonlinear systems. When controllability holds, constructive motion-planning approaches, including the homotopy-continuation method and path-lifting formulation~\cite{sussmann1992new,chitour2006continuation}, accordingly rely on the surjectivity of the differential of the endpoint map along a reference or homotopy path. We also note methods based on nilpotent approximations and interpolation entropy~\cite{jean2014nonholonomic,boizot2012motion,sussmann1992new}.
	
	This controllability and motion planning discussion is directly linked to energy minimization~\cite{boizot2012motion}. A fixed-horizon minimum-energy problem is an endpoint-constrained optimization problem: reachability is required for feasibility, while first-order optimality involves the adjoint of the endpoint map differential. In the LTI/LTV case, these ingredients collapse into the classical controllability Gramian inverse formula. In the nonlinear setting, one must instead use trajectory-dependent Gramian-like objects and coercivity conditions encoding endpoint-map surjectivity. Related energy viewpoints appear in nonlinear balancing and model reduction~\cite{scherpen1993balancing}. PMP-based methods characterize minimum-energy control through adjoint equations~\cite{boltyanski1960maximum}; in numerical synthesis, these necessary conditions are often treated as Hamiltonian boundary-value problems and solved by shooting or continuation methods~\cite{trelat2012aerospace,bonnard2011smooth}.
		Sastry and Montgomery~\cite{sastry1993structure}, for instance, differentiate PMP conditions for minimum-energy steering; in the autonomous drift-free case, they obtain that the minimum-energy control $u$ solves $\dot u=\Omega(t)u$ with $\Omega(t)^\top=-\Omega(t)$, so $|u(t)|$ is constant. Our fixed-point formula recovers this structure in that special case, while for drifted and time-dependent control-affine systems, it produces extra forcing terms, so the constant-magnitude property is not valid in general---as already noticed in~\cite{sastry1993structure}. 
		 Thus, the aim here is constructive Gramian synthesis, not only structural PMP characterization.
	
	Accordingly, we revisit the fixed-horizon endpoint-constrained
		problem of minimizing $\frac12\int_{t_0}^{T}|u(t)|^2\,dt$, subject to
		$x(t_0)=x^0$ and $x(T)=x^1$, through
		\emph{trajectory-dependent Gramians} for general control-affine
		dynamics
		\begin{equation}\tag{\textrm{$\Sigma$}}\label{eq:state-dependent-input}
			\dot{x}(t) = N_t(x(t)) + B_t(x(t))u(t),\quad x(t_0)=x^0,
		\end{equation}
		where $t\in[t_0,T]$, $x(t)\in\R^d$, $u(t)\in\R^k$ with $k\le d$,
		$N_t$ is a possibly nonautonomous drift, and $B_t$ is a time- and
		state-dependent input matrix.
	
	This article builds on the flow-conjugate representation and symmetric-Gramian steering map developed in~\cite{tamekue2025controlanalysis}, but addresses a different question: fixed-horizon $L^2$-energy minimization on feasible coercive classes. The new contributions are: (i) a Lagrange-multiplier fixed-point map, (ii) the mixed Gramian matrix, which is generally non-symmetric and is forced by the optimality equations, (iii) an existence and uniqueness result for synthesis fixed points with Picard convergence under a self-map condition and an additional iterate-contractivity condition, and (iv) a comparison with the earlier symmetric-Gramian steering control, including coincidence under an orthogonality condition. 
	Furthermore,  
	leveraging uniform complete controllability results for LTV systems~\cite{batista2017relaxed}, we provide a sufficient uniform Lie-bracket condition to verify controllability for time-dependent planar nonlinear control-affine systems with scalar controls, thereby providing regimes where the hypotheses of the minimum-energy synthesis can be verified. 
	Numerical examples, together with a companion implementation study~\cite{yu2026toolbox}, validate the approach and document its computational realization.
	
	The rest of the article is organized as follows. 
	Section~\ref{s:GA and FR} states the standing assumptions on $N_t$ and  $B_t$ in~\eqref{eq:state-dependent-input}. In Section~\ref{s:NVC}, we derive a solution representation tailored to control synthesis and introduce the associated trajectory-dependent Gramian maps. Section~\ref{s:FCC and MES} is devoted to minimum-energy control synthesis.
	Section~\ref{s:Appli and numerics} presents some applications and numerical illustrations. Section~\ref{s:conclusion} summarizes the results. Selected technical proofs are deferred to Appendix~\ref{s:proofs of technical results}. 

	\section{General assumptions and flow representation}\label{s:GA and FR}
	
	We introduce general assumptions on $N_t$ and $B_t$. 
	
	\begin{assumption}\label{ass:on N_t}
		The nonautonomous drift $N_t:[t_0, T]\times\R^d\to\R^d$ satisfies the following regularity assumptions:
		\begin{enumerate}
			\item The map $t\mapsto N_t(x)$ is $L^\infty$, for every fixed $x\in\R^d$.
			\item The map $x\mapsto N_t(x)$ is $C^2$, for every fixed $t\in[t_0, T]$.
		\end{enumerate}
		Additionally, for some constants $\Lambda_1\ge0$ and $\Lambda_2\ge0$, $\|D_xN_t(x)\|\le\Lambda_1$ and $\|D_x^2N_t(x)\|\le\Lambda_2$ for all $(t,x)\in[t_0, T]\times\R^d$. Here, $D_xN_t(x)$ and $D_x^2N_t(x)$ denote the first and second derivatives of $N_t$ with respect to $x$, respectively.
	\end{assumption}
	
	\begin{assumption}\label{ass:on B_t}
		The input matrix $B_t:[t_0, T]\times\R^d\to\R^{d\times k}$ belongs to $L^\infty((t_0, T)\times\R^d; \R^{d\times k})$. Furthermore, for a fixed $t\in[t_0, T]$, $x\mapsto B_t(x)$ is $C^2$, 
		and for some universal constants $L_B\ge0$ and $L_{DB}\ge0$, it holds $\|D_xB_t(x)\|\le L_B$ and $\|D_x^2B_t(x)\|\le L_{DB}$ for all $(t,x)\in[t_0, T]\times\R^d$.
	\end{assumption}
	
	
	\begin{remark}
		For a fixed $t$, $x\mapsto N_t(x)$ is not required to be bounded.
		Assumption~\ref{ass:on N_t} matches~\cite[Assumption~1.2]{tamekue2025controlanalysis}, while Assumption~\ref{ass:on B_t} is slightly more conservative than~\cite[Assumption~3.1]{tamekue2025controlanalysis} due to the global $C^2$ requirement and the uniform bound on $D_x^2B_t$.
		%
		These uniform bounds are mainly a technical convenience,
		and all derivative bounds are to be understood either globally on $\R^d$ or locally on a compact set containing the trajectories of interest over $[t_0, T]$. 
	\end{remark}
	
	
	
	
	\begin{notations}
		Throughout the following, we fix $t_0\ge 0$, $T>t_0$, $i\in\{1,2\}$ and $\tau_i\ge 0$, where $\tau_1=t_0$ and $\tau_2=T$. For ease in notations, we also introduce $\delta t_0:=T-t_0$, $\delta\tau_i:=T-\tau_i$, $\cX:=L^2((t_0, T); \R^k)$, $\cY:=L^\infty((t_0, T); \R^k)$ and $\|B\|_\infty:=\sup_{t\in[t_0,T]}\sup_{x\in\R^d}\|B_t(x)\|$.
		We use $|x|$ for the Euclidean norm of $x\in\R^n$ and $e_\ell$ the $\ell$th canonical basis vector of $\R^n$.
	\end{notations}

	By Assumption~\ref{ass:on N_t}, $N_t$ generates a nonautonomous flow $t\in[t_0, T]\mapsto\Phi_{t_0,t}$ and $(\Phi_{s, t})_{(s,t)\in[t_0, T]^2}$  forms a two-parameter family of diffeomorphism s.t. for any $x^0\in\R^d$, $\Phi_{s,t}(x^0)$ is the solution of (see, e.g., \cite[Chapter 2]{bressan2007introduction}) $\partial_t\Phi_{s,t}(x^0) = N_t\big(\Phi_{s,t}(x^0)\big)$ with $\Phi_{s,s}(x^0) = x^0$ for $s\le t$.
	We refer to~\cite[Section 2]{tamekue2025controlanalysis} for further algebraic identities satisfying $\Phi_{s,t}$, and \emph{a priori} general estimates involving its first $D\Phi_{s,t}(x^0)$ and second $D^2\Phi_{s,t}(x^0)$ derivative.

	\section{Representation of solutions and Gramians}\label{s:NVC}
	
	Unlike the Alekseev--Gr\"obner formula, our representation is a flow-conjugate integral form unconditionally equivalent to~\eqref{eq:state-dependent-input}, and thus a genuine nonlinear analog of the classical linear variation-of-constants formula. It fits naturally within the broad chronological calculus framework of~\cite{agravcev1979exponential}. The following statement is proven in~\cite[Theorem 3.2]{tamekue2025controlanalysis}.
	
	
	\begin{lemma}\label{thm:fb representation}
		For all $(x^0,u)\in\R^d\times \cY$, system~\eqref{eq:state-dependent-input} admits a unique absolutely continuous solution $x_u\in C^0([t_0, T]; \R^d)$ that can be represented by
		\begin{equation}\label{eq:fb representation}
			\begin{gathered}
				x_u(t)=\Phi_{\tau_i,t}\left(\Phi_{t_0,\tau_i}(x^0)+I_u(t)\right),\quad\forall t\in[t_0, T],\\
				I_u(t):=\int_{t_0}^{t}\!\!\! D\Phi_{s,\tau_i}\big(x_u(s)\big)B_s(x_u(s))u(s)\,ds.
			\end{gathered}
		\end{equation}
	\end{lemma}
	
	
	For fixed $x^0\in\R^d$ and control $u\in\cY$, let $x_u(\cdot):=x_{u,x^0}(\cdot)$ be the corresponding trajectory of~\eqref{eq:state-dependent-input} that can be represented by~\eqref{eq:fb representation}. Let $R_u(t,s)\in C^0([t_0, T]^2;\R^{d\times d})$ be the state-transition matrix (STM) of the linearized equation
	\begin{equation*}\label{eq:linearized equation B=0}
		\dot{y}(t) = D_x\big[N_t(x_u(t))+B_t(x_u(t))u(t)\big] y(t),\;t\in[t_0, T]
	\end{equation*}
	satisfying $R_u(s,s)=\idty$. Then, (see~\cite[Lemma 3.4]{tamekue2025controlanalysis})
	\begin{equation}\label{eq:factorization gen}
		R_u(t,s) = D\Phi_{t,\tau_i}(x_u(t))^{-1}\!M_{\tau_i,u}(t,s)\,D\Phi_{s,\tau_i}(x_u(s))
	\end{equation}
	for $(s,t)\in[t_0, T]$. Here, $M_{\tau_i,u}(t,s)$ is an STM of an ODE that we refer to~\cite[Lemma 3.4]{tamekue2025controlanalysis} for details. The proof of the following estimate uses the standard Gr\"onwall's lemma.
	\begin{lemma}\label{lem:estimate of the STM R_u}
		Let $(x^0,u)\in\R^d\times\cY$. It holds
		\begin{equation}\label{eq:estimate of the STM R_u}
			\hspace{-0.35cm}\begin{gathered}
				\|R_u(t,s)\|\le e^{\Lambda_{1,B}|t-s|},\qquad \forall(t,s)\in[t_0,T]^2,\\
				\|D_xR_u(t,s)\|\le\frac{\Lambda_{2,DB}}{\Lambda_{2,B}}\big(e^{2\Lambda_{1,B}|t-s|}-e^{\Lambda_{1,B}|t-s|}\big),
			\end{gathered}
		\end{equation}
		Here, $D_xR_u$ is the differential of $R_u$ w.r.t. the state, and
		\[
		\Lambda_{1,B}:=\Lambda_1+L_B\|u\|_\infty,\quad\Lambda_{2,DB}:=\Lambda_2+L_{DB}\|u\|_\infty.
		\]
	\end{lemma}
	
	For a fixed $x^0\in\R^d$, we let the endpoint map 
	\begin{equation}\label{eq:endpoint map}
		\cE:=\cE_{x^0,T}:u\in\cY\subset\cX\longmapsto\R^d,\quad\cE(u) = x_u(T).
	\end{equation}
	Next, for a fixed $u\in\cY$, we define $L_{u,\tau_i}\in\mathscr{L}(\cX,\R^d)$ by
	\begin{equation}\label{eq:map L_u tau_i}
		\hspace{-0.1cm}L_{u,\tau_i}v = \int_{t_0}^T\!\!\!D\Phi_{t,\tau_i}(x_{u}(t))\,B_t(x_u(t))\,v(t)\,dt,\quad v\in\cX,
	\end{equation}
	and for any fixed $x^1\in\R^d$ and $y_i:=\Phi_{T,\tau_i}(x^1)-\Phi_{t_0,\tau_i}(x^0)$, the feasible map
	\begin{equation}\label{eq:feasible map}
		F_{\tau_i}:u\in\cY\subset\cX\mapsto\R^d,\quad F_{\tau_i}(u) = L_{u,\tau_i}u-y_i.
	\end{equation}
	
	Their differentials $D\cE(u),\,DF_{\tau_i}(u):\cX\to\R^d$ satisfy (see~\cite[Lemmas 3.12]{tamekue2025controlanalysis} and~\cite[Theorem~3.2.6]{bressan2007introduction} )
	\begin{equation}\label{eq:identities}
		\begin{gathered}
			DF_{\tau_i}(u)= D\Phi_{T,\tau_i}\!\big(x_u(T)\big)\,D\cE(u),\\
			D\cE(u)v=\int_{t_0}^T\!\!\!R_u(T,t)B_t(x_u(t))v(t)\,dt,\quad v\in\cX.
		\end{gathered}
	\end{equation}
	Let $L_{u,\tau_i}^\ast,\,DF_{\tau_i}(u)^\ast\in\mathscr{L}(\R^d,\cX)$ denote the adjoint of $L_{u,\tau_i}$ and $DF_{\tau_i}(u)$, respectively. It holds for all $z\in\R^d$,
	\begin{equation}\label{eq:adjoint}
		\begin{split}
			[L_{u,\tau_i}^\ast z](t) &=  B_t(x_u(t))^\top D\Phi_{t,\tau_i}(x_{u}(t))^\top z,\\
			[DF_{\tau_i}(u)^\ast z](t) &=  B_t(x_u(t))^\top Q_{u,\tau_i}(T,t)^\top z
		\end{split}
	\end{equation}
	with $Q_{u,\tau_i}(T,t):=D\Phi_{T,\tau_i}(x_u(T))R_u(T,t)$.
	
	
		
		
		
	\begin{definition}\label{def:Gramians}
		We define the following matrices for $u\in\cY$:
		
		(1)\textit{``Endpoint Gramian''}: $\cM_{\tau_i}(u) = DF_{\tau_i}(u)DF_{\tau_i}(u)^\ast$.
		
		(2) \textit{``Representation Gramian''}: $\cN_{\tau_i}(u) = L_{u,\tau_i}L_{u,\tau_i}^\ast$.
		
		(3) \textit{``Mixed Gramian''}: $\cG_{\tau_i}(u) = L_{u,\tau_i}DF_{\tau_i}(u)^\ast$.
	\end{definition}
	Note that
	$\cM_{\tau_i}(u)$ and $\cN_{\tau_i}(u)$ are symmetric positive semi-definite (PSD) matrices, $\cG_{\tau_i}(u)$ is not in general. The roles are distinct: $\cM_{\tau_i}$ measures local surjectivity of the feasible map, $\cN_{\tau_i}$ generates the steering synthesis of~\cite{tamekue2025controlanalysis}, while $\cG_{\tau_i}$ is imposed by the Lagrange multiplier equations for the minimum-energy problem. Its non-symmetry is therefore structural, not a numerical design choice. If $N_t\equiv A_t$ and $B_t(x)=B_tx$, $A_t\in\R^{d\times d}$, $B_t\in\R^{d\times k}$, then
	\begin{equation*}
		\hspace{-0.3cm}\cM_{\tau_i}(u)=\cN_{\tau_i}(u)=\cG_{\tau_i}(u)=\int_{t_0}^T\!\!\!\!R(\tau_i,t)B_tB_t^\top R(\tau_i,t)^\top\!dt
	\end{equation*}
	where $R(t,s)$ is the STM of $\dot x=A_tx$ with $R(s,s)=\idty$.
	
If $\cG_{\tau_i}(u)$ is invertible at $u\in\cY$, then both $L_{u,\tau_i}$ and $DF_{\tau_i}(u)$ are onto, and consequently both $\cN_{\tau_i}(u)$ and $\cM_{\tau_i}(u)$ are invertible. The converse need not hold, because even if $L_{u,\tau_i}$ and $DF_{\tau_i}(u)$ are both onto, the composition $L_{u,\tau_i}DF_{\tau_i}(u)^\ast$ may be singular.
	
	When $\cN_{\tau_i}(u)$, respectively $\cM_{\tau_i}(u)$, is invertible at $u\in\cY$, we denote by $R_{u,\tau_i}$, respectively $S_{u,\tau_i}$, the canonical right-inverse, as given by
	\begin{equation*}\label{eq:canonical right inverses}
		\hspace{-0.25cm} R_{u,\tau_i}:=L_{u,\tau_i}^\ast\cN_{\tau_i}(u)^{-1},\;\, S_{u,\tau_i}:=DF_{\tau_i}(u)^\ast\cM_{\tau_i}(u)^{-1}.
	\end{equation*}
		
		\begin{proposition}\label{pro:preparatory}
			Let $u\in\cY$.
			If $\cN_{\tau_i}(u)$ is invertible, set
			$A_{u,\tau_i}:=K_{u,\tau_i}R_{u,\tau_i}$, $K_{u,\tau_i}:=DF_{\tau_i}(u)-L_{u,\tau_i}$. Then
			\begin{equation}\label{eq:link Gramians N}
				\cG_{\tau_i}(u)=\cN_{\tau_i}(u)\left(\idty+A_{u,\tau_i}\right)^\top.
			\end{equation}
			In particular, if $ \idty+A_{u,\tau_i}$ is invertible,
			then $\cG_{\tau_i}(u)$ is invertible.
			
			If $\cM_{\tau_i}(u)$ is invertible, set
			$S_{u,\tau_i}:=DF_{\tau_i}(u)^\ast\cM_{\tau_i}(u)^{-1}$ and
			$B_{u,\tau_i}:=K_{u,\tau_i}S_{u,\tau_i}$, $K_{u,\tau_i}:=DF_{\tau_i}(u)-L_{u,\tau_i}$. Then
			\begin{equation}\label{eq:link Gramians M}
				\cG_{\tau_i}(u)=\left(\idty-B_{u,\tau_i}\right)\cM_{\tau_i}(u).
			\end{equation}
			In particular, if $\idty-B_{u,\tau_i}$ is invertible,
			then $\cG_{\tau_i}(u)$ is invertible.
		\end{proposition}
	
	\noindent The proof of Proposition~\ref{pro:preparatory} is immediate, and it is omitted for brevity. See, for instance, \cite[Lemma 3.12]{tamekue2025controlanalysis}.
	
	\section{Main Results: minimum-energy control}\label{s:FCC and MES}
	
	Our goal in this section is to synthesize a control that minimizes energy within the feasible set. 
		
		Recall that for $x^0\in\R^d$, we say that $x^1\in\R^d$ is \emph{reachable on} $[t_0,T]$ from $x^0$ if there is $u\in\cY$ s.t. the  solution $x_u(\cdot)$ of \eqref{eq:state-dependent-input} satisfies $x_u(T)=x^1$. The reachable set is given by
		\begin{equation}\label{eq:reachable set nonautonomous}
			\cR_{\Sigma}(T, x^0) := \Big\{\, x_u(T)\;:\; x_u(\cdot)\ \text{solves \eqref{eq:state-dependent-input}}\Big\}.
		\end{equation}
		If $\cR_{\Sigma}(T, x^0)=\R^d$ for every $x^0\in\R^d$, one says that \eqref{eq:state-dependent-input} is \textit{completely controllable (here after controllable)} on $[t_0, T]$.
	
	\begin{definition}[\cite{tamekue2025controlanalysis}]
		Let $C_i>0$ depends only on system data (e.g., $\|B\|_\infty$, $L_B$, $\Lambda_1$, $\Lambda_2$, $\delta t_0$, and possibly $|x^0|$). For $\Pi\in\{\cN_{\tau_i},\cM_{\tau_i}\}$, a \textit{feasible coercive class} is a set
		\begin{equation}\label{eq:gd feasible coercive class}
			\cF(C_i):=\{u\in\cY:\lambda_{\min}(\Pi(u))\ge C_i^{-1}\}.
		\end{equation}
	\end{definition}
	
	The set $\cF(C_i)$ may be empty unless there exists $u^\flat\in\cY$ with
	$\lambda_{\min}(\Pi(u^\flat))\ge C_i^{-1}$. It is closed in $\cY$ since
	$u\mapsto\Pi(u)$ and $u\mapsto\lambda_{\min}(\Pi(u))$ are continuous. 
	Thus, nonemptiness of $\cF(C_i)$ is a scope assumption, not an automatic consequence of controllability. 
		In applications, $C_i$ is obtained from an analytic lower bound, as in~\eqref{eq:coercivity bound full actuated case} and  Proposition~\ref{pro:ubg_2d_Mu} below.

Motivated by Proposition~\ref{pro:preparatory}, we also introduce the following subset, where $\pi(u)\in\{\idty+A_{u,\tau_i},\idty-B_{u,\tau_i}\}$,
\begin{equation}\label{eq:finite-dim invertibility}
	\cI(c_i):=\{u\in\cF(C_i):\sigma_{\min}(\pi(u))\ge c_i^{-1}\}
\end{equation}
for some $c_i>0$ which depends only on the system's data. Here, $\sigma_{\min}(\pi(u))$ is the minimum singular value of $\pi(u)$. It is therefore clear that for some $\gamma_i>0$, it holds that
\begin{equation}
	\sigma_{\min}(\cG_{\tau_i}(u))\ge\gamma_i^{-1} ,\qquad\forall u\in\cI(c_i).
\end{equation}

For the unicycle example in Section~\ref{sss:unicycle}, $\gamma_i$ is determined directly as a function of $C_i$. This suggests that one can often work directly in a unified uniformly invertible class. 

Recall the Gramian-like synthesis map in~\cite{tamekue2025controlanalysis}, 
\begin{equation}\label{eq:gramian-like synthesis map}
	\cS_i(u) = L_{u,\tau_i}^\ast\,\cN_{\tau_i}(u)^{-1}y_i,\quad \forall u\in\cF(C_i)
\end{equation}
which, with $\zeta_{y_i}:= C_i\,\|B\|_\infty\,e^{\Lambda_1\delta t_0}\,|y_i|$, satisfies
\begin{equation}\label{eq:zeta radius}
	\|\cS_i(u)\|_\infty\le\zeta_{y_i}\quad\forall u\in\cF(C_i).
\end{equation}
Accordingly, define
\begin{equation}\label{eq:feasible coercive ball}
	\cF_{\zeta_{y_i}}:=\cI(c_i)\cap\{u\in\cY:\|u\|_\infty\le \zeta_{y_i}\}.
\end{equation}

	\begin{lemma}\label{lem:minimizer-is-fp}
		Assume that $\cI(c_i)\neq\emptyset$. If $\bar u$ is a local minimizer of $\tfrac12\|u\|_{L^2}^2$ over
		$\mathfrak F_{y_i}:=\{u\in\cF_{\eta_{y_i}}:\ F_{\tau_i}(u)=0\}$, then $\bar u$ satisfies the fixed-point relation
		\begin{equation}\label{eq:a priori fixed point from LM}
			\bar u= DF_{\tau_i}(\bar u)^\ast\,\cG_{\tau_i}(\bar u)^{-1}y_i.
		\end{equation}
	\end{lemma}

\noindent The proof of Lemma~\ref{lem:minimizer-is-fp} is given in Section~\ref{s:minimum-energy control}. This motivates the Lagrange multiplier synthesis map,
\begin{equation}\label{eq:synthesis map Z_i}
	\cZ_i(u)=DF_{\tau_i}(u)^\ast\,\cG_{\tau_i}(u)^{-1}y_i,\qquad \forall u\in\cI(c_i).
\end{equation}
If $\bar u_i$ is a fixed point of $\cZ_i$, as given by~\eqref{eq:a priori fixed point from LM}, then
\begin{equation}\label{eq:a priori estimate fp of Z_i}
	\|\cZ_i(\bar u_i)\|_\infty=\|DF_{\tau_i}(\bar u_i)^\ast\,\cG_{\tau_i}(\bar u_i)^{-1}y_i\|_\infty\le\zeta_{y_i}.
\end{equation}
\begin{lemma}\label{lem:a priori estimate on Z_i(u)}
	There exists a positive constant $\Xi_i:=\Xi_i(C_i,\zeta_{y_i},\Lambda_1,\Lambda_2,\|B\|_\infty,\delta t_0,L_B)>0$ such that
	\begin{equation}\label{eq:a priori estimate on Z_i(u)}
		\|\cZ_i(u)\|_\infty\le\vartheta_{y_i}:=\Xi_i|y_i|,\qquad\forall u\in\cF_{\zeta_{y_i}}.
	\end{equation}
\end{lemma}

\noindent The proof of Lemma~\ref{lem:a priori estimate on Z_i(u)} uses the standard Gr\"onwall lemma and Lemma~\ref{lem:estimate of the STM R_u}. 

It follows from~\eqref{eq:a priori estimate on Z_i(u)} that $\cZ_i$ maps $\cI(c_i)\cap\{\|u\|_\infty\le\zeta_{y_i}\}$ into $\{\|u\|_\infty\le\vartheta_{y_i}\}$. Consistent with the unconstrained controls assumption, we introduce a unified radius
\begin{equation}\label{eq:unified radius}
	\eta_{y_i}\ge\max\{\zeta_{y_i},\vartheta_{y_i}\},
\end{equation}
and the unified ball
\begin{equation}\label{eq:unified feasible coercive ball}
	\cF_{\eta_{y_i}}:=\cI(c_i)\cap\{u\in\cY:\|u\|_\infty\le \eta_{y_i}\}
\end{equation}
such that $\cZ_i$ maps $\cF_{\eta_{y_i}}$ into $\{\|u\|_\infty\le\eta_{y_i}\}$. So, the following invariance condition ensures that $\cZ_i$ is a self-map on $\cF_{\eta_{y_i}}$,
\begin{equation}\label{eq:self-map condition Zi-coercivity}
	\cZ_i\!\left(\cI(c_i)\right)\subset \cI(c_i).
\end{equation}

	\begin{lemma}\label{lem:abstract-fixed-point}
		Assume that $\cI(c_i)\neq\emptyset$ and that~\eqref{eq:self-map condition Zi-coercivity} holds. Then there exists a positive sequence $(q_m)_{m\ge1}$ with $q_m:=K_i^m\vartheta_m(p_i)$ such that for every $m\in\N$, it holds that
		\begin{equation}\label{eq:Zi-iterate-contractivity}
			\|\cZ_i^{\,m}(u)-\cZ_i^{\,m}(v)\|_\infty
			\le q_{m}\|u-v\|_\infty,
			\quad \forall u,v\in\cF_{\eta_{y_i}},
		\end{equation}
		where $K_i>0$ depends only on system data, and for some $p_i\in(0,1)$ and $\vartheta_m(p_i)\to 0$ as $m\to\infty$. Furthermore, if there exists $m_i\ge1$ such that $q_{m_i}<1$, then $\cZ_i$ admits a unique fixed point $u_i\in\cF_{\eta_{y_i}}$, and the Picard iteration
		$u^{(m+1)}=\cZ_i(u^{(m)})$ converges to $u_i$ for any $u^{(0)}\in\cF_{\eta_{y_i}}$.
	\end{lemma}

The proof follows the same mechanism as in~\cite[Theorem~3.10]{tamekue2025controlanalysis} and it is omitted for brevity.

The main result of this article is the following.

	\begin{theorem}[Main Theorem]\label{thm:main controllability result nonlinear}
		Let $(x^0,x^1)\in\R^d\times\R^d$ and set $y_i:=\Phi_{T,\tau_i}(x^1)-\Phi_{t_0,\tau_i}(x^0)$.
		Assume that $\cI(c_i)\neq\emptyset$, that~\eqref{eq:self-map condition Zi-coercivity} holds, and that the iterate-contractivity condition of Lemma~\ref{lem:abstract-fixed-point} is satisfied. Let $\bar u_i\in\cF_{\eta_{y_i}}$ be the resulting unique fixed point of $\cZ_i$.
		Then $\bar u_i$ steers~\eqref{eq:state-dependent-input} from $x^0$ to $x^1$ on $[t_0,T]$ and satisfies
		\begin{equation}\label{eq:control with G_i}
			\bar u_i(t)=B_t(x_{\bar u_i}(t))^\top Q_{\bar u_i,\tau_i}(T,t)^\top\,\cG_{\tau_i}(\bar u_i)^{-1}y_i.
		\end{equation}
		Moreover, one has the energy identity
		\begin{equation}\label{eq:energy certificate G_i}
			\|\bar u_i\|_{L^2}^2= z_i^\top\cM_{\tau_i}(\bar u_i)z_i,
			\quad z_i:=\cG_{\tau_i}(\bar u_i)^{-1}y_i.
		\end{equation}
	Finally, $\bar u_i$ is the unique global minimizer of $\tfrac12\|u\|_{L^2}^2$ over the feasible set $\mathfrak F_{y_i}:=\{u\in\cF_{\eta_{y_i}}:\ F_{\tau_i}(u)=0\}$,
	that is,
	\[
	\|\bar u_i\|_{L^2}^2\le \|w\|_{L^2}^2\quad\forall\,w\in\mathfrak F_{y_i},
	\; \text{with equality iff }\;w=\bar u_i\;\text{a.e}.
	\]
	\end{theorem}
\begin{remark}
	Formula~\eqref{eq:control with G_i} is consistent with the PMP equations derived in~\cite{sastry1993structure}. Under the additional
		time-regularity required to differentiate, setting
		$g_\ell(t,x):=B_t(x)e_\ell$ and
		$p_i(t):=Q_{\bar u_i,\tau_i}(T,t)^\top
		\cG_{\tau_i}(\bar u_i)^{-1}y_i$, one has
		$\bar u_{i,\ell}(t)=\langle p_i(t),g_\ell(t,x_{\bar u_i}(t))\rangle$
		and hence
	\[
		\dot{\bar u}_{i,\ell}
		=
		\sum_{m=1}^k
		\langle p_i,[g_m,g_\ell]\rangle\bar u_{i,m}
		+
		\langle p_i,\partial_tg_\ell+[N_t,g_\ell]\rangle
	\]
	where $[f,h]$ denote the Lie-bracket of $f$ and $h$. Thus, when $N_t\equiv0$ and $g_\ell$ is autonomous, the forcing
		term vanishes and $\dot{\bar u}_i=\Omega(t)\bar u_i$ with
		$\Omega^\top(t)=-\Omega(t)$, so $|\bar u_i(t)|$ is constant.
	
\end{remark}

\begin{proof}[Proof of Theorem~\ref{thm:main controllability result nonlinear}]
	Since $\bar u_i$ is the fixed point of $\cZ_i$, we have
	$\bar u_i=DF_{\tau_i}(\bar u_i)^\ast z_i$, with
	$z_i=\cG_{\tau_i}(\bar u_i)^{-1}y_i$. Hence
	\[
	L_{\bar u_i,\tau_i}\bar u_i
	=L_{\bar u_i,\tau_i}DF_{\tau_i}(\bar u_i)^\ast z_i
	=\cG_{\tau_i}(\bar u_i)z_i
	=y_i,
	\]
	and therefore $\bar u_i$ steers~\eqref{eq:state-dependent-input} from $x^0$ to $x^1$. Formula~\eqref{eq:control with G_i} follows from the expression of $DF_{\tau_i}(\bar u_i)^\ast$. Moreover,
	\[
	\|\bar u_i\|_{L^2}^2
	=\langle DF_{\tau_i}(\bar u_i)^\ast z_i,DF_{\tau_i}(\bar u_i)^\ast z_i\rangle_{L^2}
	=z_i^\top\cM_{\tau_i}(\bar u_i)z_i,
	\]
	which proves~\eqref{eq:energy certificate G_i}.
	
	To prove global minimality, \cite[Proposition~3.14]{tamekue2025controlanalysis} guarantees the existence of at least one minimizer $\bar u\in\mathfrak F_{y_i}$ of $\tfrac12\|u\|_{L^2}^2$. Any such minimizer is, in particular, a local minimizer. By Lemma~\ref{lem:minimizer-is-fp}, it satisfies the fixed-point relation $\bar u=\cZ_i(\bar u)$. By Lemma~\ref{lem:abstract-fixed-point} and~\eqref{eq:Zi-iterate-contractivity}, the fixed point of $\cZ_i$ in $\cF_{\eta_{y_i}}$ is unique under the additional iterative-contractivity condition; hence $\bar u=\bar u_i$. This proves that $\bar u_i$ is the unique global minimizer on $\mathfrak F_{y_i}$.
\end{proof}

\begin{remark}\label{rem:energy-gap}
	Let $\bar u_i\in\cF_{\eta_{y_i}}$ be the minimum-energy steering control from Theorem~\ref{thm:main controllability result nonlinear}, and let $u_i\in\cF_{\eta_{y_i}}$ be the steering control from~\cite[Theorem~3.15]{tamekue2025controlanalysis}, as given by
	\begin{equation}\label{eq:almost optimal control}
		u_i=\cS_i(u_i):=L_{u_i,\tau_i}^\ast\,\cN_{\tau_i}(u_i)^{-1}y_i.
	\end{equation}
	Then $\|\bar u_i\|_{L^2}^2\le \|u_i\|_{L^2}^2$
	with equality iff $u_i=\bar u_i$. In particular, they coincide under the orthogonality condition in~\cite[Theorem~3.13]{tamekue2025controlanalysis}, and in particular in the linear case.
\end{remark}

\section{Examples}\label{s:Appli and numerics}

We develop examples that enable a comparative assessment of the matrices
$\cM_{\tau_2}=\cM_T$, $\cN_{\tau_2}=\cN_T$, and $\cG_{\tau_2}=\cG_T$ in Definition~\ref{def:Gramians}. Section~\ref{ss:applications} introduces the model forms considered, verifies the feasibility of the control synthesis, and finally, Section~\ref{ss:NI} numerically enacts the synthesis. 

\subsection{Models considered and verification of control feasibility}\label{ss:applications}

\subsubsection{Planar control-affine systems with scalar inputs}\label{sss:planar control-affine with scalar inputs}

We specialize in $d=2$ and $k=1$. We work under Assumption~\ref{ass:on N_t} with in addition $\partial_tDN_t(\cdot)\in\cY:=L^\infty((t_0,T);\R^2)$, and $B_t(x)\equiv B_t$ is state-independent with $B_t,\dot B_t,\ddot B_t\in\cY$.

\begin{proposition}\label{pro:ubg_2d_Mu}
	Assume
	\begin{equation}\label{eq:ubg_2d_bracket}
		\inf_{t\in[t_0,T]}\inf_{x\in\R^2}
		\big|\det\big(B_t,\ DN_t(x)B_t-\dot B_t\big)\big|\ge \beta>0.
	\end{equation}
	Fix $x^0\in\R^2$ and $U>0$. There is $\alpha_{\cM_T}(\beta,|x^0|,U)>0$ s.t.
	\begin{equation}\label{eq:coercivity_Mu_bounded}
		\cM_T(u)\succeq \alpha_{\cM_T}\idty,\quad
		\forall u\in L^\infty(t_0,T),\; \|u\|_\infty\le U.
	\end{equation}
\end{proposition}

	The proof of Proposition~\ref{pro:ubg_2d_Mu} is given in Section~\ref{ss:proof of ubg_2d}. Thus, fixing $x^0\in\R^2$ and $U>0$, and setting $C_2:=\alpha_{\cM_T}(\beta,|x^0|,U)^{-1}$ with the feasible class in~\eqref{eq:gd feasible coercive class} defined using $\Pi=\cM_T$, one has
	\[
	\cU_U:=\{u\in L^\infty(t_0,T):\|u\|_\infty\le U\}
	\subset \cF(C_2).
	\]
	Whenever the self-mapping, uniform invertibility set~\eqref{eq:finite-dim invertibility} is nonempty, and iterate-contractivity conditions in Theorem~\ref{thm:main controllability result nonlinear} are verified, Theorem~\ref{thm:main controllability result nonlinear} applies.

\subsubsection{Fully-actuated control-affine systems}\label{sss:FA system}

Throughout, we consider~\eqref{eq:state-dependent-input} with $k=d$. We work under Assumptions~\ref{ass:on N_t} and~\ref{ass:on B_t} with the pointwise invertibility estimate
\begin{equation}\label{eq:not pointwise invertibility of B_t}
	\sigma_{\min}(B_t(x)) \ge\ b(t),\quad\text{for a.e. } t\in(t_0,T),\,\forall x\in\R^d
\end{equation}
for some nonzero $b\in L^1((t_0,T);\R_+)$.

Under~\eqref{eq:not pointwise invertibility of B_t}, \cite[Proposition~3.19]{tamekue2025controlanalysis} ensures that
	\begin{equation}\label{eq:coercivity bound full actuated case}
		\cN_T(u) \succeq {e^{-2\Lambda_1\delta t_0}\,\|b\|_1^2}{\delta t_0}^{\!\!-1},\quad \forall (x^0,\,u)\in\R^d\times\cY.
	\end{equation}

By~\eqref{eq:coercivity bound full actuated case}, the feasible coercive class is nonempty in the fully actuated setting. Consequently, Theorem~\ref{thm:main controllability result nonlinear} applies under hypothesis~\eqref{eq:not pointwise invertibility of B_t} whenever the uniform invertibility set~\eqref{eq:finite-dim invertibility} is nonempty and iterate-contractivity condition is verified.

For steering alone,~\eqref{eq:not pointwise invertibility of B_t} also yields
\begin{equation}\label{eq:feedback linearization fullactuation}
	u_{\mathrm{fl}}(t)=B_t(x(t))^{-1}\!\left({(x^1-x^0)}/{\delta t_0}-N_t(x(t))\right),
\end{equation}
used as a baseline in Fig.~\ref{fig:control_comparison}. Thus, here the added value is the minimum-energy guarantee.

\subsubsection{Unicycle kinematic model}\label{sss:unicycle}

Consider the unicycle
\begin{equation}\label{eq:unicycle}
	\dot p_1=v\cos\theta,\quad\dot p_2=v\sin\theta,\quad\dot\theta=\omega,\;\, t\in[0,T],
\end{equation}
where $ x=(p,\theta)\in\R^2\times\S^1$. Eq.~\eqref{eq:unicycle} is controllable on $[0,T]$; see~\cite{siciliano2009robotics}. Direct computation yields 
\begin{equation}\label{eq:unicycle important}
	\cN_T(u)=\begin{pmatrix}G_\omega&0\\0&T\end{pmatrix},\quad
	\cG_T(u)=\begin{pmatrix}G_\omega&0\\ \delta p(T)^\top&T\end{pmatrix},
\end{equation}
where $G_\omega:=H_\omega(0)$, and we let
\[
\begin{gathered}
	H_\omega(t):=\int_t^T\!\!\! g_\omega(s)g_\omega(s)^\top ds,\qquad
	g_\omega(s):=\binom{\cos\theta(s)}{\sin\theta(s)},\\
	\delta p(T):=J\int_0^T\!\!\!(p(t)-p(T))\,dt,\quad J:=\begin{pmatrix}
		0 & 1\\
		-1 & 0
	\end{pmatrix}.
\end{gathered}
\]
Moreover, one has
\begin{equation}\label{eq:unicycle important eig}
	\lambda_{\min}(G_\omega)
	=\frac{T-\left|\int_0^T e^{2i\int_0^t\omega(s)\,ds}\,dt\right|}{2}
\end{equation}
so that $\cN_T(u)$ and $\cG_T(u)$ are invertible iff
\begin{equation}\label{eq:set I of invertibility unicycle}
	u\in\left\{(v,\omega)^\top\in\cY:\omega\neq 0\text{ in }L^1(0,T)\right\}.
\end{equation}

For $x^k=(p^k,\theta^k)$, $k=0,1$, 
the minimum-energy control $\bar u_2=(\bar v_2,\bar\omega_2)$ from $x^0$ to $x^1$, as given by~\eqref{eq:control with G_i} recasts as
	\begin{equation}\label{eq:MEC unicycle}
		\begin{gathered}
			\bar v_2(t) = g_{\bar\omega_2}(t)^\top G_{\bar\omega_2}^{-1}\delta p,\\
			\bar\omega_2(t) = \delta p^\top\begin{bmatrix}
				G_{\bar\omega_2}^{-1}(H_{\bar\omega_2}(t)-\bar H_{\bar\omega_2})JG_{\bar\omega_2}^{-1}
			\end{bmatrix}\delta p+\frac{\delta\theta}{T},
		\end{gathered}
	\end{equation}
	where $\bar H_\omega$ is the average of $H_\omega$ on $[0,T]$, $\delta p:=p^1-p^0$ and $\delta\theta:=\theta^1-\theta^0$. Note that $\bar v_2$ is explicitly given by~\eqref{eq:MEC unicycle} once $\bar\omega_2$ is obtained as a fixed point of the induced scalar map.
In particular, Sastry--Montgomery's structure~\cite{sastry1993structure} is recovered,
\[
	\dot{\bar u}_2(t)
	=
	\bigl(g_{\bar\omega_2}(t)^\top
	JG_{\bar\omega_2}^{-1}\delta p\bigr)J\bar u_2(t),
	\qquad\frac{d}{dt}|\bar u_2(t)|^2=0.
\]
Likewise, the steering control~\eqref{eq:almost optimal control} recasts explicitly as
	\begin{equation}
		\hspace{-0.25cm}v_2(t) = g_{\omega_2}(t)^\top G_{\omega_2}^{-1}\delta p,\;\,\begin{cases}
			\omega_2(t) = {\delta\theta}/{T},\;\,\delta\theta\neq 0,\\
			\int_0^T\!\!\omega_2(t)\,dt=0,\;\,\delta\theta=0.
		\end{cases}
	\end{equation}

 Note that~\eqref{eq:unicycle} also admits feedback-linearization steering controls $(v_{\mathrm{fl}},\omega_{\mathrm{fl}})$ from $(p^0,\theta^0)$ to $(p^1,\theta^1)$, namely
	\begin{equation}\label{eq:feedback-linearization unicycle}
		v_{\mathrm{fl}}(t) = \pm|\dot p(t)|,\quad\omega_{\mathrm{fl}}(t) = {\langle\dot p(t),Ja(t)\rangle}/{|\dot p(t)|^2 },
	\end{equation}
	where $a(t)=(a_1(t),a_2(t))^\top\in\R^2$ is the linear Gramian-based control, steering $\ddot p(t) = a(t)$ from $p^0\in\R^2$ to $p^1\in\R^2$. The control $u_{\mathrm{fl}}=(v_{\mathrm{fl}},\omega_{\mathrm{fl}})$ is used as a baseline in Fig.~\ref{fig:control_comparison}.

Finally, PMP/Jacobi-elliptic minimum-energy formulas for~\eqref{eq:unicycle} appear in~\cite{kim2021minimum}, for a free terminal heading.  Here, the angular target is prescribed through a chosen lift of $\delta\theta$.

\subsection{Numerical illustrations}\label{ss:NI}
This section presents numerical results supporting our framework. The simulations use adaptive Dormand--Prince integration, numerical quadrature for the trajectory-dependent Gramians, and standard linear solvers for the resulting $d\times d$ systems. Extended algorithmic and scalability details are provided in the companion manuscript~\cite{yu2026toolbox}.

\begin{figure}[t]
\centering
\begin{tikzpicture}[scale=0.75]
	\begin{groupplot}[
		group style={
			group size=1 by 2,
			vertical sep=1.1cm
		},
		width=\columnwidth,
		height=4.7cm,
		xmin=0.5, xmax=6.5,
		xtick={1,2,3,4,5,6},
		xticklabels={P1,P2,3D-1,3D-2,U1,U2},
		x tick label style={font=\footnotesize},
		tick label style={font=\footnotesize},
		label style={font=\footnotesize},
		ymajorgrids=true,
		grid style={dashed,gray!30},
		every axis plot/.append style={only marks, mark size=2.3pt},
		legend style={
			font=\footnotesize,
			at={(0.5,1.18)},
			anchor=south,
			legend columns=3,
			draw=none,
			/tikz/every even column/.append style={column sep=0.35cm}
		}
		]
		
		\nextgroupplot[
		ylabel={$\frac12\|u\|_{L^2}^2$},
		ymode=log,
		log basis y={8},
		ymin=1, ymax=900,
		ytick={1,8,64,512},
		yticklabels={$8^0$,$8^1$,$8^2$,$8^3$},
		]
		
		\draw[gray!55] (axis cs:2.5,1) -- (axis cs:2.5,512);
		\draw[gray!55] (axis cs:4.5,1) -- (axis cs:4.5,512);
		\draw[gray!55] (axis cs:6.5,1) -- (axis cs:6.5,512);
		
		\node[font=\scriptsize] at (axis cs:1.5,670) {Pendulum};
		\node[font=\scriptsize] at (axis cs:3.5,670) {3D RNN};
		\node[font=\scriptsize] at (axis cs:5.5,670) {Unicycle};
		
		\addplot+[mark=*, black]
		coordinates {
			(1.00,23.78) (2.00,76.61) 
			(3.00,65.91) (4.00,10.22) (5.00,1.16) (6.00,98.75)
		};
		\addlegendentry{$\bar u_2$}
		
		\addplot+[mark=diamond*, black!65]
		coordinates {
			(1.00,24.11) (2.00,76.66) 
			(3.00,68.13) (4.00,10.60) (5.00,1.17) (6.00,98.75)
		};
		\addlegendentry{$u_2$}
		
		\addplot+[mark=triangle*, black!35]
		coordinates {
			(1.00,32.67) (2.00,85.60)
			(3.00,82.09) (4.00,129.92)
			(5.00,1.34) (6.00,367.00)
		};
		\addlegendentry{$u_{\mathrm{fl}}$}
		
		
		\nextgroupplot[
		ylabel={$\|u\|_{L^\infty}$},
		ymode=log,
		ymin=1, ymax=1000,
		ytick={1,10,100,1000},
		yticklabels={$10^0$,$10^1$,$10^2$,$10^3$},
		]
		
		\draw[gray!55] (axis cs:2.5,1) -- (axis cs:2.5,1000);
		\draw[gray!55] (axis cs:4.5,1) -- (axis cs:4.5,1000);
		\draw[gray!55] (axis cs:6.5,1) -- (axis cs:6.5,1000);
		
		\addplot+[mark=*, black]
		coordinates {
			(1.00,9.96) (2.00,26.56) 
			(3.00,16.07) (4.00,3.81) (5.00,1.49) (6.00,62.85)
		};
		
		\addplot+[mark=diamond*, black!65]
		coordinates {
			(1.00,9.16) (2.00,26.30) 
			(3.00,15.07) (4.00,4.34) (5.00,1.31) (6.00,62.83)
		};
		
		\addplot+[mark=triangle*, black!35]
		coordinates {
			(1.00,12.84) (2.00,25.58)
			(3.00,10.68) (4.00,9.50)
			(5.00,2.04) (6.00,461.45)
		};
		
		
	\end{groupplot}
\end{tikzpicture}
\caption{Comparison of the steering controls $\bar u_2$, $u_2$, and $u_{\mathrm{fl}}$ over six transfer problems. In all cases, the error $\|x_u(T)-x^1\|_2$ is negligible, ranging from $10^{-8}$ to $10^{-15}$. We denote by P1--P2 the two pendulum transfers, by 3D-1--3D-2 the two 3D RNN transfers, and by U1--U2 the two unicycle transfers.}
\label{fig:control_comparison}
\end{figure}


\subsubsection{Torque-controlled pendulum with varying length}\label{ss:TCP with vl}

We consider the nondimensional equation of a torque-controlled pendulum with varying length~\cite[Section~2]{belyakov2009pendulum},
\begin{equation}\label{eq:TCP-abstract}
N_t(x)=\begin{bmatrix}x_2\\-a(t)\sin x_1-\gamma(t)x_2\end{bmatrix},
\;\,
B_t(x)=\begin{bmatrix}0\\b(t)\end{bmatrix},
\end{equation}
where
$
b(t)=\bigl(1+0.5\cos(t)\bigr)^{-2}$, $a(t)=\lambda^2\sqrt{b(t)}$,
$\gamma(t)=- b(t)\sin(t)+\beta\lambda$ with $\lambda=0.78$ and $\beta=0.13$.

For all $(t,x)\in[t_0,T]\times\R^2$, it holds that
\[
\big|\det\big(B_t,\ DN_t(x)B_t-\dot B_t\big)\big|=b(t)^2\ge 2^{-4}
\]
and Proposition~\ref{pro:ubg_2d_Mu} applies. Thus, our framework yields the associated minimum-energy control under further standing assumptions. 

Finally, system~\eqref{eq:TCP-abstract} is feedback-linearizable via
\begin{equation}\label{eq:feedback-linearization}
u_{\mathrm{fl}}(t)=({v(t)+\gamma(t)x_2(t)+a(t)\sin x_1(t)})/{b(t)},
\end{equation}
where $v$ is the linear Gramian-based control steering the double integrator $\ddot\theta=v$ from $(\theta^0,\omega^0)^\top$ to $(\theta^1,\omega^1)^\top$. Thus, $u_{\mathrm{fl}}$ provides a steering baseline and is used in Fig.~\ref{fig:control_comparison}.

\subsubsection{Fully-actuated 3D recurrent neural network (RNN)}\label{sss:3DHRNN with scalar input}
We present in Fig.~\ref{fig:control_comparison} simulations of the fully actuated $3$D RNN
with $B_t(x)\equiv\idty$. Namely,
\begin{equation}\label{eq:3D example HRNN}
\dot x(t) = N(x(t))+u(t),\quad x(0) = x^0,\quad t\in[0,T],
\end{equation}
where $N(x):=-Dx+W\sigma(x)$. We use the parameters
\begin{equation*}\label{eq:parameters values for simulations 3DHRNN}
\begin{gathered}
	D = \operatorname{diag}(1.25,1.5,1),\quad W = \begin{pmatrix}
		3 & 1 & -0.5\\
		2 & 1 & 0.5\\
		0 & -1.5 & 1.25
	\end{pmatrix},\\
	\sigma_j(s)=(1+e^{-s})^{-1},\;\sigma(x) = (\sigma_j(x_j))_{1\le j\le 3}^\top.
\end{gathered}
\end{equation*}
The analysis of Section~\ref{sss:FA system} applies to this example.

\section{CONCLUSIONS}\label{s:conclusion} 

We developed a minimum-energy synthesis for possibly nonautonomous control-affine systems with time- and state-dependent input matrices. Relative to~\cite{tamekue2025controlanalysis}, the new element is the Lagrange-multiplier fixed-point map and the mixed Gramian that characterize the $L^2$-minimum over feasible coercive classes. This differs from specialized steering methods such as feedback linearization, which may produce explicit controls when applicable but does not by itself solve the original fixed-horizon minimum energy problem, and from MPC, which relies on a chosen discretization and numerical optimization scheme. The present approach gives an analytic Gramian-based fixed-point synthesis with an intrinsic energy certificate under verifiable uniformly invertibility and self-map conditions; uniqueness and Picard convergence are certified by an additional iterate-contractivity check.

Extending the coercivity framework to high-dimensional underactuated systems is an open mathematical challenge, even as empirical computation scales effectively~\cite{yu2026toolbox}.





\appendices

\section{Additional Proofs}\label{s:proofs of technical results}


\subsection{Proof of Lemma~\ref{lem:minimizer-is-fp}}\label{s:minimum-energy control}
Since $\cG_{\tau_i}(\bar u)=L_{\bar u,\tau_i}DF_{\tau_i}(\bar u)^\ast$ is invertible, $DF_{\tau_i}(\bar u)$ is onto. Therefore, the Lagrange multiplier theorem in Hilbert spaces~\cite[Theorem 43.D, p.~290]{zeidler2013nonlinear} applied to $\tfrac12\|u\|_2^2$ and the submersion $F_{\tau_i}$ gives $\lambda\in\R^d$ such that
\begin{equation}\label{eq:u bar}
	\bar u=DF_{\tau_i}(\bar u)^\ast\lambda,
	\qquad
	L_{\bar u,\tau_i}\bar u=y_i.
\end{equation}
Left-multiplying the first identity in~\eqref{eq:u bar} by $L_{\bar u,\tau_i}$ and using the second one yields $\lambda=\cG_{\tau_i}(\bar u)^{-1}y_i$.
Substituting back gives $\bar u=DF_{\tau_i}(\bar u)^\ast\cG_{\tau_i}(\bar u)^{-1}y_i=\cZ_i(\bar u)$.
\qed

\subsection{Proof of Proposition~\ref{pro:ubg_2d_Mu}}\label{ss:proof of ubg_2d}

\begin{proof}
We apply~\cite[Theorem~5]{batista2017relaxed}. Recall that
\[
\cM_T(u)=\int_{t_0}^T R_u(T,t)B_tB_t^\top R_u(T,t)^\top\,dt,
\]
where $R_u(t,s)$ is the state-transition matrix of $\dot y(t)=A_u(t)y(t)$ with $A_u(t):=DN_t(x_u(t))$. Define

\[
\begin{gathered}
M_0(t):=B_t,\qquad M_1(t):=A_u(t)B_t-\dot B_t,\\
\cC_u(t):=\begin{bmatrix}M_0(t)&M_1(t)\end{bmatrix}.
\end{gathered}
\]

Using the uniform bounds on $DN_t$, $\partial_tDN_t$, $D^2N_t$, $B_t$, $\dot B_t$, and $\ddot B_t$, one obtains
\[
\|\cC_u(t)\|\le C_0,\qquad \operatorname{Lip}(t\mapsto\mathcal C_u(t))\le C_{|x^0|,U},
\]
for some constants $C_0>0$ and $C_{|x^0|,U}>0$, uniformly for all $u\in L^\infty(t_0,T)$ with $\|u\|_\infty\le U$. Here, the Lipschitz estimate uses, for some constant $V_{|x^0|,U}>0$, the bound
\[
|\dot x_u(t)|\le V_{|x^0|,U},\qquad t\in[t_0,T],
\]
by Gr\"onwall lemma.
Moreover,~\eqref{eq:ubg_2d_bracket} yields
\[
\sigma_{\min}(\cC_u(t))
=\frac{|\det \cC_u(t)|}{\sigma_{\max}(\cC_u(t))}
\ge \frac{\beta}{C_0}=:\alpha,
\]
and therefore $\cC_u(t)\cC_u(t)^\top\succeq \alpha^2\idty$ for all $t\in[t_0,T]$.

Thus $t\mapsto\cC_u(t)$ is uniformly bounded, uniformly Lipschitz, and uniformly pointwise nondegenerate on $\{u\in L^\infty(t_0,T):\|u\|_\infty\le U\}$. Applying~\cite[Theorem~5]{batista2017relaxed} gives $\cM_T(u)\succeq \alpha_{\cM_T}\idty$
for some $\alpha_{\cM_T}(\beta,|x^0|,U)>0$.  
\end{proof}


\end{document}